\documentclass[letterpaper, 10 pt, journal, twoside]{IEEEtran}

\usepackage[noadjust]{cite}

\usepackage{blindtext}
\usepackage{graphicx}
\usepackage{amssymb, amsmath, amsthm}
\usepackage{amsfonts}
\usepackage{color}
\usepackage{enumitem}

\usepackage{stmaryrd}
\usepackage{bbm}

\newtheorem{theorem}{Theorem}
\newtheorem{corollary}{Corollary}[theorem]

\newtheorem{lemma}[theorem]{Lemma}

\begin{document}

\title{Contraction and Robustness of Continuous Time Primal-Dual Dynamics}
\author{Hung D. Nguyen, Thanh Long Vu, Konstantin Turitsyn, Jean-Jacques Slotine \thanks{HN is with Nanyang Technological University; TLV, JJS, and KT are with the Department of Mechanical Engineering, Massachusetts Institute of Technology, Cambridge, MA. Work of KT was funded by DOE/GMLC 2.0 project: ``Emergency monitoring and controls through new technologies and analytics'', NSF CAREER award 1554171, and 1550015. HN was supported by NTU SUG, Siebel Fellowship, and VEF.}
}

\maketitle
\thispagestyle{empty}
\pagestyle{empty}
                       
\begin{abstract}      
The Primal-Dual (PD) algorithm is widely used in convex optimization to determine saddle points. While the stability of the PD algorithm can be easily guaranteed, strict contraction is nontrivial to establish in most cases. This work focuses on continuous, possibly non-autonomous PD dynamics arising in a network context, in distributed optimization, or in systems with multiple time-scales. We show that the PD algorithm is indeed strictly contracting in specific metrics and analyze its robustness establishing stability and performance guarantees for different approximate PD systems. We derive estimates for the performance of multiple time-scale multi-layer optimization systems, and illustrate our results on a primal-dual representation of the Automatic Generation Control of power systems.

\end{abstract}

\begin{IEEEkeywords}                           
Primal-dual dynamics, continuous optimization, strict contraction, robustness, hierarchical architecture
\end{IEEEkeywords}      


\section{Introduction}

\IEEEPARstart{M}{ulti-scale} optimization is ubiquitous in both natural and artificial systems. Multiple time-scales have long been viewed as a fundamental organizing principle in the modular architecture and evolution of complex systems ~\cite{simon1991architecture}. In engineering, such layering provides a powerful design architecture for decomposing multiple stage decision processes in networked infrastructure~\cite{chiang2007layering}. In cyber-physical networks such as the power grid, the structure of the dynamic equations is strongly constrained by physical laws. However, as a number of recent works have shown, e.g., \cite{NaliPD, zhao2014design, dorfler2016breaking}, the nature of these equations admits a natural optimization based perspective. To a large extent, the underlying ideas behind the approaches exploited in those works can be traced back to the classical saddle point problem, and its associated primal-dual (PD) algorithm. 

The saddle-point problem, first considered in the context of market equilibrium in economics, appears when the system is simultaneously minimizing a function over one set of its variables and maximizing it over the other variable set. Due to the unique characteristic of the cost function, namely, being convex in the first set variables while being concave in the second, the primal-dual algorithm consists of implementing gradient descent compatibly with the convexity/concavity of the cost function. This algorithm was introduced in the classic work of Arrow, Hurwicz, and Uzawa \cite{arrow1958studies}. Recently, the body of literature on this algorithm has been rapidly growing due to its efficiency in decentralized optimization in network applications \cite{Paganini, LowAGC, simpson2016input, zhao2016unified}. Similar saddle point problems also appear naturally in the context of machine learning, e.g., in  support vector machine representations~\cite{kosaraju2018primal}
and in the adversarial training of deep networks~\cite{madry2017towards}. 

\subsection{The PD algorithm in a distributed context}
In its simplest form, the primal dual (PD) dynamics minimizes the function $g(x) \in \mathbb{R}^n \to \mathbb{R}$ subject to a set of linear or nonlinear constraints $h(x) = 0$ with $h(x) \in \mathbb{R}^n \to \mathbb{R}^m$. The corresponding Lagrangian function is given by 
\begin{equation}\label{eq:L}
\ell(x,\nu) = g(x) + \nu^T h(x),
\end{equation}
with $\nu \in \mathbb{R}^m$ being the dual variables. We denote the full state of the system by $z = (x,\nu) \in \mathbb{R}^{n+m}$. The continuous time primal dual algorithm defines a dynamic system in $x(t)$ and $\nu(t)$ described by
\begin{align}
 \dot{x} &= - \partial_x \ell = -\partial_x g - (\partial_x h)^T \nu \\
 \dot{\nu} &= +\partial_\nu \ell = h
\end{align}


There are two common applications of primal-dual dynamics most frequently discussed in the literature. First, it can be naturally applied to design distributed continuous time optimization systems using the Lagrangian relaxation type of approach. Within this formulation, the large-scale optimization problem is represented in a superposition form $g(x) = \sum_k g_k (\mathcal{X}_k)$ with $\mathcal{X}_k$ denoting the set of variables for subsystem $k$. Coupling between the subsystems is enforced by constraining subsets of variables in different subproblems to be equal to each other. These additional constraints are represented as $h(x) = E^T x$, with $E$ being an incidence matrix of a directed graph defining equivalencies between variable replicas. In this setting, every subproblem is coupled to common dual variables but not other subproblem variables directly. Every subproblem can then be solved by an individual agent, while the dual dynamics gradually adjusts the dual variables until the equivalence constraints are satisfied.

Another typical application arises in network flows of natural or artificial nature characterized by some conservation laws. Assuming that some subset of variables represents the flows of the conserved quantities, the conservation laws are expressed as $h(x) = E x - q$, with $E$ being the incidence matrix reflecting the topology of interconnection of individual subsystems, and $q$ the vector of external source/sink flows. In equilibrium, the total flow on every node of the interconnection is balanced. However, during the transient dynamics violation can occur due to non-zero $\dot{\nu}$. These violations may be interpreted as the accumulation of the transported quantity on the node storage elements. In traditional circuits, the $\dot{\nu}$ terms represent charging of effective node capacitance. 
Whenever such a formulation of physical equations exists, one can naturally solve the distributed optimization problem complementing the Lagrangian with additional terms representing the objective of the controller.

Unless otherwise stated here we assume the constraints of the form $h(x) = Ex-q$, which leads to the following formulation 
\begin{subequations} \label{eq:PD}
\begin{align} \label{eq:primal}
\dot{x} &= -\partial_x g - E^T\nu \\
\dot{\nu} &= Ex - q
\end{align}
\end{subequations}


\subsection{Contributions}
In section \ref{sec:strongcontr}, we establish the strict contraction of continuous PD dynamics in the form given by \eqref{eq:PD}, by constructing explicitly contracting metrics and estimating the corresponding contraction rates. 
Next, in Section \ref{sec:app}, we derive several new results concerning the robustness of the PD systems including the bounds on the long-term steady-state errors induced by various types of disturbances, inaccurate estimations, and multiple time-scale optimization. We dedicate Section \ref{sec:sim} to present a relevant power systems example where we present the AGC problem in a PD form then illustrate the error bounds derived in Section \ref{sec:app}.
Note that our results can apply to non-autonomous PD dynamics with general objective functions, while most recent related work deals with autonomous PD dynamics \cite{WangConvex} or quadratic objective functions \cite{SimpsonQuadratic}. 
The ISS analysis in \cite{CherukuriSaddle}, which characterizes error bounds to fixed saddle points, is relevant to the estimates derived in this work. However, we quantify errors relative to the time-varying  instantaneous optimum (Lemma \ref{lem:zpdzstar}), as well as those relative to the time-varying  trajectory induced by imperfect measurements (Lemma \ref{le:unobservable_to_pd}).



\section{strict contraction of PD dynamics} \label{sec:strongcontr}
Throughout this paper, we use $\|\cdot\|$ to denote the norm in which the considered system is strictly contracting. Similarly, $\mu(A)$ denotes the matrix measure of $A$ corresponding to the discussed norm. In particular, for the $2$-norm, one has $\mu(A) = \lambda_\mathrm{max}(A+A^T)/2$ where $\lambda_\mathrm{max}$ is the maximal eigenvalue. 

We proceed by reviewing the basics of contraction analysis for nonlinear dynamical systems. For holistic descriptions on this topic, see \cite{lohmiller1998contraction, slotine2003modular}. Let us consider a nonlinear dynamical system $\dot{z} = f(z, t)$ where $f$ is a continuous and sufficiently smooth function of the state variable $z$. The infinitesimal dynamics can be given as $\dot{\delta z} = \partial_z f \delta z$. Contraction analysis can characterize the dynamics of the distance between two close trajectories in some weighted two norm defined as: $\|\delta z\| = \|\delta y\|_2$ where we introduce a differential coordination transformation $\delta y = \Theta \delta z$ with an invertible metric transformation $\Theta$. The rate of change of the distance can be calculated accordingly $d \|\delta z\|^2/dt = 2\delta y^T F \delta y$ with $F = \dot{\Theta}\Theta^{-1} + \Theta \partial_z f \Theta^{-1}$ being the generalized Jacobian. Whenever the symmetric part of the generalized Jacobian is uniformly negative definite, i.e., there exists $\beta > 0$ s.t. $\mu(F) \leq -\beta < 0$, the system is (strictly) $\beta$-contracting and all trajectories will converge exponentially towards each other with a contraction rate larger or equal to $\beta$.

The basic contraction property of the primal dual dynamics with respect to the identity metric is simple to establish \cite{slotine2003modular}. Specifically, consider a virtual displacement between the two close trajectories characterized by the vector $\delta z= \left(\delta x, \delta \nu\right)$. These displacement vectors are described by:
\begin{subequations}
\begin{align} \label{eq:displacementdyn}
\delta  \dot{x} &= - H \delta x - E^T \delta \nu \\
\delta  \dot{\nu} &= E\delta x
\end{align}
\end{subequations}
with $H = \partial_{xx} f$ being the Hessian of the objective function. We further assume that the objective function is strictly convex in $x$, thus its corresponding Hessian is positive definite $H \succ 0$.
One can easily see that the primal-dual dynamics described by \eqref{eq:PD} is contracting with respect to the traditional Euclidean norm: $\|\delta z\|_2^2 =  \delta x^T \delta x + \delta \nu^T \delta \nu$:
\begin{equation} \label{eq:contraction-identity}
 \frac{d}{dt}\|\delta z\|_2^2 = - 2 \delta x^T H\delta x \leq 0
\end{equation}
Among other things, this result implies that the distance $r(z(t),z^\star) = \|z(t)-z^\star\|$ between the current point and any equilibrium satisfying the KKT conditions is a non-increasing function, a well-known result since the classical works on PD dynamics \cite{arrow1958studies}. Moreover, it establishes the connection between this original result and more recent approach to PD systems via Krasovskii type Lyapunov functions \cite{Paganini}. Indeed, the existence of a Krasovskii Lyapunov function implies contraction of the system, which in turn verifies that the distance between any point and any equilibrium only decreases.

However, this system is not strictly contracting; there may be close trajectories that don't eventually get closer to each other in Euclidean metrics. This is indeed the case for trajectories with the same $x$ but different $\nu$. Moreover, whenever the matrix $E$ is not full row rank, the optimum of the original system may not be unique, and the system may converge to different equilibria. In this case, there is no strict contraction because the distance between two trajectories starting from different equilibria does not change. However, for systems with full rank $E$ the question arises, whether strict contraction of the trajectories can be established in some other metric. 

\subsection{Strict contraction of PD dynamics}
In the following we shall develop a metric which certifies strict contraction of PD systems. Consider a coordinate transformation $\delta y = \Theta \delta z$ with invertible ``skew'' metric transformation $\Theta$:
\begin{equation}
    \Theta = \begin{bmatrix}
    I & \alpha E^T\\
    0 & (I-\alpha^2 E E^T)^{\frac{1}{2}}
    \end{bmatrix}
\end{equation}
where $I$ denotes the identity matrix of appropriate dimensions. 
Then, for the two neighbour trajectories with the virtual displacement $\delta z$, we introduce the distance $\|\delta z\|_\alpha^2 = \delta y^T \delta y = \delta z^T\, \Theta^T\Theta \,\delta z$. Following the arguments from the previous sections we arrive at the generalized Jacobian $F = -\Theta^{-T} Q \Theta^{-1}$ with
\begin{equation}
    Q = \begin{bmatrix}
    H - \alpha E^TE & \frac{\alpha}{2} EH \\
    \frac{\alpha}{2} H E^T & \alpha EE^T
    \end{bmatrix}
\end{equation}
Theorem below establishes sufficient conditions for strict contraction of the primal-dual algorithm in this metric.
\begin{lemma} \label{lem:strong_alpha}
If $\alpha$ satisfies
\begin{equation}\label{eq:alphacondition}
      0 < \alpha < \frac{1}{\max\left\{\|E\|_2, \|EH^{-1/2}\|_2^2 + \|H\|_2/4\right\}}
 \end{equation}
the primal-dual system defined in \eqref{eq:PD} is strictly contracting in metric $\Theta$, with rate $\beta = \mu_2\left(-\Theta^{-T}Q \Theta^{-1}\right)$.
 \end{lemma}
\begin{proof} The condition $\alpha < 1/\|E\|$ ensures not only that the metric transformation $\Theta$ is invertible but also the distance $\|\delta z\|_\alpha^2$ is a positive definite form:
\begin{align}
    \|\delta z\|_\alpha^2 
    &= \delta z^T \Theta^T \Theta \delta z \nonumber\\ 
    &=\|\delta x + \alpha E^T\delta\nu\|^2_2 + \delta \nu^T (1-\alpha^2 E E^T)\delta \nu. \label{eq:quadform}
\end{align}
The rate of change of this form, according to \eqref{eq:PD} is given by
\begin{align}
    &\frac{d}{dt}\|\delta z\|_\alpha^2 \nonumber\\
    &=-2\delta x^T (H - \alpha E^T E)\delta  x -  2\alpha\delta \nu^T EH\delta x  -2\alpha \delta \nu^T E E^T\delta \nu \nonumber\\
    &= -2\delta x^T (H - \alpha E^T E - \frac{\alpha}{4} H^2)\delta  x -2\alpha \|E^T\delta \nu + \frac{1}{2}H \delta x\|^2_2 \nonumber
\end{align}
Matrix $H - \alpha E^T E - \frac{\alpha}{4} H^2$ is positive definite whenever $\alpha$ satisfies \eqref{eq:alphacondition}, and the system is therefore strictly contracting. For the considered system, the rate of change of the distance $\|\delta z\|_\alpha^2 = \|\Theta \delta z\|_2^2$ is given by $\frac{d}{dt} \|\delta z\|_\alpha^2 = - 2\,\delta z^T Q \delta z$, $Q \succ 0$. The guaranteed contraction rate is given by $\beta = \lambda_{\textrm{min}}\left(\Theta^{-T}Q \Theta^{-1}\right)$ which is easily proved by noting that $-\delta z^T Q \delta z  = - (\Theta \delta z)^T \left(\Theta^{-T}Q \Theta^{-1}\right) (\Theta \delta  z) \leq -  \lambda_{\textrm{min}}\left(\Theta^{-T}Q \Theta^{-1}\right) \, \|\Theta \delta z\|_2^2$.
\end{proof}
Note that, for linear and autonomous systems, the metric $\Theta^T \Theta$ above is equivalent to the strict quadratic Lyapunov functions for PD dynamics presented in \cite{SimpsonQuadratic}. However, the main advantage of our approach stems from the fact that contraction analysis applies to a more general class of systems which can be non-autonomous. Many important classes of problems are non-autonomous in nature including the optimization framework considered in this work.

\section{Applications} \label{sec:app}
While the PD dynamics is an extremely flexible framework applicable to a broad variety of continuous time optimization problems, its perfect realization is not viable in most of the practical situations. In this section, we analyze stability and performance of quasi-PD systems that approximate the ``true'' PD dynamics. Throughout the section, we adopt a number of assumptions and definitions reviewed below. We consider a system characterized by the Lagrangian $\ell(x,\nu,t)$ of the form \eqref{eq:L} and the ``true'' PD dynamics expressed compactly as 
\begin{equation} \label{eq:truePD}
 \dot{z}^{\mathrm{pd}} = f(z^{\mathrm{pd}},t)
\end{equation}
where $z^{\mathrm{pd}} = (x^\mathrm{pd}, \nu^{\mathrm{pd}})$. We assume that the system is contracting with respect to the uniform metric associated with the norm $\|z\| = \|\Theta z\|_2$ with some constant matrix $\Theta$ as presented in Section \ref{sec:strongcontr}. We assume that the system is strictly $\beta$-contracting. 

While many of the results can be easily extended to a more general case of non-uniform metric $\Theta$ explicitly depending on $z$, for the sake of simplicity we restrict the discussion only to the uniform case. In the following, we show that the contraction rate with respect to such a metric can be naturally used to characterize the performance of more realistic approximately primal-dual systems.

In many practical situations, the continuous time PD algorithm is utilized in a non-stationary setting where the system is subject to constantly changing external conditions. In this case, the PD dynamics allows the system to track the optimal operating condition which also changes in time. For example, in power system context, the secondary frequency control can keep the system close to the optimal power flow solutions as the external factors such as the load power consumption or renewable generation change. Typically there is a time-scale separation between the fast PD dynamics and slow changes of external parameters. In this case, the deviations from the optimum are small enough and can be safely ignored. However, in a more general context establishing rigorous bounds on the deviations from the optimum is essential for certifying the safety and performance of the systems.

In the following, we assume, that the Lagrangian $\ell(x,\nu,t)$ is explicitly dependent on time, and characterize this dependence implicitly through the position of the instantaneous optimum $z^\star(t) = (x^\star(t), \nu^\star(t))$. Also, we explicitly assume that the rate of change $\|\dot{z}^\star\|$ of such instantaneous optimum is bounded. Then strict contraction of the PD dynamics provides us with a natural for quantification of the deviations from the optimum.



\begin{lemma} \label{lem:zpdzstar}
Consider now the distance $r(t) = \|z^\mathrm{pd}(t)-z^\star(t)\|$ between the state $z^\mathrm{pd}(t)$ and the instant optimum $z^\star(t)$. This distance satisfies the following differential inequality:
\begin{equation}
 \dot r \leq - \beta r + \|\dot{z}^\star\|
\end{equation}
\begin{proof}
This result is proven by direct differentiation of $r^2(t)$ and application of the contraction property. The term $\beta r $ represents the contraction of the fixed equilibrium system, while the term  $\|\dot{z}^\star\|$ accounts for the motion of the instantaneous equilibrium point $z^\star(t)$.
\end{proof}
\end{lemma}
\begin{corollary} \label{co:pd_to_star}
In the steady state $t\to \infty$, whenever the rate of equilibrium point movement is bounded the ``true'' PD system above is confined to the ball 
\begin{equation}\label{eq:zzs}
\|z^\mathrm{pd}(t) - z^\star(t)\| \leq \frac{1}{\beta}\sup_{t}\|\dot{z}^\star\|
\end{equation}
\end{corollary}


\subsection{Robustness of PD systems}
Next, we consider the systems that deviate from the ``true'' PD dynamics. Our primary motivation is the multi-scale optimization system where the decisions are made by different layers on multiple time-scales (for example, see \cite{del2013contraction, bousquet2015contraction}). It is common in this setting for the higher layers of the hierarchy to have limited observability of the lower layer states. Most commonly, the optimization logic on higher layers either relies on imperfect observations or assumes that the faster lower layers have already equilibrated. To model this setting, we represent the approximate primal-dual dynamics by
\begin{equation}
 \dot{z} = f(z,t) + d(z,t)
\end{equation}
where $d(z,t) = f(\hat{z}, t) - f(z, t)$ represents the substitution of true signal $z = (z_\mathrm{o}, z_\mathrm{u})$ with its estimate $\hat{z}= (z_\mathrm{o}, \hat{z}_\mathrm{u})$. In this section we assume that the function $f$ is Lipschitz, so that the following two inequalities hold: 
\begin{align}
  \|d(z,t)\| = \|  f(\hat{z}, t) - f(z, t)\| &\leq  \xi\|\hat{z} - z\| \label{eq:lip1}\\
 \|f(z,t)\|  = \|f(z,t) - f(z^\star,t)\| &\leq \eta \|z-z^\star\| \label{eq:lip2}
\end{align}

We now prove the following intermediate lemma (see also \cite{slotine1991applied} and well as \cite{bousquet2015contraction}).
\begin{lemma}\label{le:unobservable_to_pd}
In steady-state, the distance $ r^\mathrm{pd}(t) = \|z(t) - z^\mathrm{pd}(t)\| $ between the approximate and true primal-dual dynamics systems satisfies the differential inequality:
\begin{equation} \label{eq:distance_to_true}
 \dot{r}^\mathrm{pd}\leq -\beta r^\mathrm{pd} + \|d\|
\end{equation}
\begin{proof}
Consider a trajectory $\tilde{z}(t)$ following the original flow, i.e. satisfying $\dot{\tilde{z}} = f(\tilde{z},t)$ and starting at time $t$ at $z(t)$, i.e. $\tilde{z}(t) = z(t)$. Trajectory $\tilde{z}(t)$ together with $z^{\mathrm{pd}}(t)$ can be considered as two individual trajectories of the same system which is contracting the at the rate of $\beta$. Strict contraction of the original system implies that the distance between these trajectories will decrease as:
\begin{equation} \label{eq:ztzpd}
\|\tilde{z}(t+dt) - z^\mathrm{pd}(t+dt)\| \leq  (1-\beta dt)\|z(t)-z^\mathrm{pd}(t)\|
\end{equation} 
where we have utilized the fact that $z(t) = \tilde{z}(t)$ at time $t$. At the same time, the distance between
$\tilde{z}(t)$ and $z(t)$ in the same interval has increased by at most $dt \|d(z(t),t)\|$. Therefore, we have that
\begin{equation} \label{eq:ztz}
 \|\tilde{z}(t+dt) - z(t+dt)\| \leq \|d(z(t),t)\| \, dt.
\end{equation}

Combining \eqref{eq:ztzpd} and \eqref{eq:ztz} via triangle inequality: $\|\tilde{z}(t+dt) - z(t+dt)\| - \|z(t+dt) - z^\mathrm{pd}(t+dt)\| \leq \|\tilde{z}(t+dt) - z^\mathrm{pd}(t+dt)\|$, and taking the limit $dt \to 0$, yields \eqref{eq:distance_to_true}.
\end{proof}
\end{lemma}

\begin{corollary} \label{co:appx_to_pd}
After exponential transients at rate $\beta > 0$, the distance between the non-ideal and ideal PD can be bounded as 
\begin{equation} \label{eq:zzpd}
\|z(t) - z^\mathrm{pd}(t)\| \leq \frac{\xi}{\beta}\sup_t\|\hat{z}(t)-z(t)\|
\end{equation}
\end{corollary}
\begin{proof}
From Lemma \ref{le:unobservable_to_pd} one has the following when the system settles down: $\|z(t) - z^\mathrm{pd}(t)\| \leq \sup_t\|d(z(t),t)\|/\beta$. Combining this relation and the inequality \eqref{eq:lip1}, yields \eqref{eq:zzpd}
\end{proof}

Next, we consider a setting where the actual system does not follow exactly the primal-dual dynamics, although the PD system is a reasonable approximation. In practice, this situation can occur for a variety of reasons, for example, due to fast degrees of freedom in plant dynamics lacking the PD structure, or due to imperfect observers introducing additional delays in the system. While both settings can be modeled in the same way, for exposition purposes we restrict the discussion only to the case of imperfect observers. We assume, that the subset of directly unobservable variables $z_\mathrm{u}$ is estimated by a separate observer system $\dot{\hat{z}}_u = f_\mathrm{u}(\hat{z}_\mathrm{u}, z_\mathrm{u})$ that satisfies the following conditions:
\begin{enumerate}
 \item A subset of the observer states, $\hat{z}_\mathrm{u}$ is an asymptotically unbiased estimate of $z_\mathrm{u}$, so that for constant $z_\mathrm{u}$, the observer converges to the manifold $f_\mathrm{u}(\hat{z}_\mathrm{u}, z_\mathrm{u}) = 0$ satisfying $\hat{z}_\mathrm{u} = z_\mathrm{u}$.
 \item Dynamics of the observer is partially contracting with respect to $\hat{z}_\mathrm{u}$ with a contraction rate of $\hat{\beta}$. 
\end{enumerate}

The following formal results allow us to characterize the performance of the system.
\begin{lemma} \label{lem:zhatbounded}
Whenever  $\hat{\beta} > \xi$ and $\|z-z^\star\|$ is upper bounded, the long-term steady state error of the observer is bounded by
\begin{equation} \label{eq:hatzz}
 \|\hat{z} - z\| \leq \frac{\eta}{\hat{\beta} -\xi} \sup_t \|z-z^\star\|
\end{equation}
\begin{proof}
Applying Lemma \ref{lem:zpdzstar} and the Lipschitz bounds we obtain
\begin{subequations}
\begin{align}
 \|\hat{z} - z\| &\leq  \frac{1}{\hat{\beta}} \sup_t \|\dot{z}_\mathrm{u}\| \label{eq:y_to_z_a}\\
 &\leq \frac{1}{\hat{\beta}}\sup_t\left(\|f(z,t)\| + \|d(t)\|\right) \label{eq:y_to_z_b} \\
 &\leq \frac{1}{\hat{\beta}}\sup_t\left( \eta \|z - z^\star\| +\xi\|\hat{z} - z\| \label{eq:y_to_z_c}\right)
\end{align}
\end{subequations}
where we have used the relation $\|\dot{z}_\mathrm{u}\| \leq \|\dot{z}\| \leq \|f(z,t)\| + \|d(t)\|$ to arrive at \eqref{eq:y_to_z_b} from \eqref{eq:y_to_z_a}. Solving for $\sup_t \|\hat{z}-z\|$ yields \eqref{eq:hatzz}.
\end{proof}
\end{lemma}
\begin{theorem} \label{theo:upbound} Given the conditions stated in Lemma \ref{lem:zhatbounded} and $\beta (\hat{\beta} -\xi) - \eta \xi >0 $, the steady state optimal tracking error is upper bounded
\begin{equation} \label{eq:z_to_star_2}
 \|z-z^\star\| \leq \frac{ \hat{\beta} -\xi}{\beta (\hat{\beta} -\xi) - \eta \xi} \sup_t \|\dot{z}^\star\|
\end{equation}
\end{theorem}

\begin{proof}
\begin{subequations}
\begin{align}
 \|z- z^\star\|  & \leq \|z- z^\mathrm{pd}\| + \|z^\mathrm{pd} - z^\star\|  \\
&\leq  \frac{1}{\beta}\sup_t\left(\xi\|\hat{z} - z\| + \|\dot{z}^\star\|\right)  \label{eq:zzstar2}\\
&\leq \frac{1}{\beta}\sup_t\left(\frac{\eta \xi}{\hat{\beta} - \xi} \|z-z^\star\| + \|\dot{z}^\star\|\right) \label{eq:zzstar3}
\end{align}
\end{subequations}
where we use relation \eqref{eq:hatzz} to arrive at the last inequality from \eqref{eq:zzstar2}. By solving for $\sup_t \|z-z^\star\|$ we obtain \eqref{eq:z_to_star_2}.
\end{proof}
Note, that in the limit of perfect observer with $\hat{\beta} \to \infty$ one recovers the bound \eqref{eq:zzs}.
\begin{corollary} \label{co:z_to_zpd}
The distance between the perturbed and ``true'' primal-dual dynamics can be bounded as
\begin{equation}
 \|z - z^\mathrm{pd}\| \leq \frac{1}{\beta}\frac{\eta \xi}{\beta (\hat{\beta} - \xi) -\eta \xi}\sup_t \|\dot{z}^\star\|
\end{equation}
\end{corollary}
Corollary \ref{co:z_to_zpd} can be easily proved by combining the results from Corollary \ref{co:appx_to_pd} and Lemma \ref{lem:zhatbounded}.

The bounds derived in this section are illustrated with a practical power system in Section \ref{sec:sim}. While their practical relevance should be assessed in the context of specific systems with well-defined structural features, the bounds are tight if one assumes that only Lipschitz constants are known. For example, tightness of the bound \eqref{eq:zzs} can be established by considering the simple scalar ODE $\dot{z}(t) = -z(t) + t$ with $z(0) = 0$. 



\subsection{Performance of layered optimization architectures}
Our results could be naturally applied to multi-layer optimization systems commonly occurring in nature and engineering. In these systems, each layer typically performs its own optimization \cite{changeux1973theory, simon1991architecture, chiang2007layering}, and interacts with other layers. Usually, the dynamics of the layers are separated in time-scales, so that the dynamics of higher levels is slower in comparison to that of lower ones. In engineering systems, the algorithms employed on the individual layers are often designed with two assumptions in mind: that the lower layer has converged to its optimal equilibrium, and that the higher layer inputs can be assumed to be constant. In this section, we study a broader class of such systems, constrained only to be Lipschitz and contracting.

Mathematically, the setting above can be expressed by introducing the subsets of PD variables $z_k$ corresponding to different layers of optimization. For notational simplicity, we assume that each layer interacts directly with two neighboring layers and the ``true'' PD dynamics is described by
\begin{equation}
    \dot{z}_k = f_k(z_{k-1}, z_k, z_{k+1}, t)
\end{equation}
From the viewpoint of layer $k$, the function $z_{k-1}$ can be considered as an exogenous factor, which affects the position of instant equilibrium $z_k^\star = z_k^\star(z_{k-1}(t), t)$. On the other hand, we assume that each layer $k$ is designed under the assumption that all the faster layers have equilibrated and so $z_{k+1}$ is replaced by $z_{k+1}^\star(z_k, t)$. In this case, following the same logic as in previous section, the actual dynamics can be represented as $\dot{z}_k = f_k + d_k$ with
\begin{equation}\label{eq:dk}
    d_k = f_k(z_{k-1}, z_k, z_{k+1}^\star(z_k, t), t) - f_k(z_{k-1}, z_k, z_{k+1}, t)
\end{equation}
\begin{theorem}
Consider a multi-layer optimization system described above. Assume that $f_k$ is Lipschitz with respect to $z_k$ and $z_{k+1}$ with constants denoted as $\eta_k, \xi_k$ respectively. Furthermore, assume that each function $z_k^\star(z_{k-1}, t)$ is also Lipschitz with respect to $z_{k-1}$, with Lipschitz constant $\rho_k$. 

Let $ \gamma_k = 1- \eta_k \tau_{k+1} \rho_{k+1}$.  For small enough Lipschitz constants,  such that  $\gamma_k  > 0 $
and $\gamma_k  \beta_k >  \xi_k$, the system is stable and its performance is characterized by inequalities  $\|z_k - z_k^\star\| \leq \tau_k \sup_t\|\dot{z}_{k}^\star\| $ with $\tau_k$ given by \eqref{eq:zzs} or \eqref{eq:z_to_star_2} for the lower layer, and for the higher ones by 
\begin{equation}\label{eq:taukfinal}
    \tau_k  = \frac{ \gamma_k}{\gamma_k  \beta_k  - \xi_k}.
\end{equation}
\end{theorem}
\begin{proof}
We start by bounding the term $d_k$. Using the definition \eqref{eq:dk} we have 
\begin{align}
\|d_k\| &\leq \eta_k \|z_{k+1} - z_{k+1}^\star\| \nonumber\\
&\leq \eta_k \tau_{k+1} \sup_t \|\dot{z}_{k+1}^\star\| \nonumber\\
&\leq \eta_k \tau_{k+1}\rho_{k+1}\|\dot{z}_k\| \nonumber\\
&\leq \eta_k \tau_{k+1}\rho_{k+1} \sup_t\left(\xi_k \|z_k - z_k^\star\| + \|d_k\|\right) \nonumber
\end{align}
Hence, whenever $\eta_k\tau_{k+1}\rho_{k+1} < 1$ we have
\begin{equation} \label{eq:dkfinal}
    \|d_k\| \leq \frac{\xi_k \|z_k - z_k^\star\|}{1-\eta_k \tau_{k+1}\rho_{k+1}}
\end{equation}
Next, as long as the system $f_k$ is contracting, we have $\|z_k - z_k^\star\| \leq \left(\|\dot{z}_k^\star\| + \|d_k\|\right)/\beta_k$. Using \eqref{eq:dkfinal}, assuming $\xi_k  < \beta_k (1-\eta_k \tau_{k+1}\rho_{k+1})$, we arrive at \eqref{eq:taukfinal}.
\end{proof}

Note, that while multi-scale PD dynamics have been a motivation for this section analysis, the results apply more broadly to general multi-scale contracting systems, not necessarily optimization ones. On the other hand, many practical iterative optimization algorithms in discrete time can be viewed as perturbed versions of the baseline continuous dynamics and thus admit analysis with the proposed techniques. Our multiple time-scale results are also applicable to the \emph{virtual} contracting systems used to analyze synchronization phenomena \cite{wang2005partial}. In particular, quorum sensing strategies \cite{tabareau2010synchronization, russo2010global} can be used to coordinate multiple dynamics.

\section{Numerical simulations} \label{sec:sim}
\begin{figure}[t]
    \centering
    \includegraphics[width=1\columnwidth]{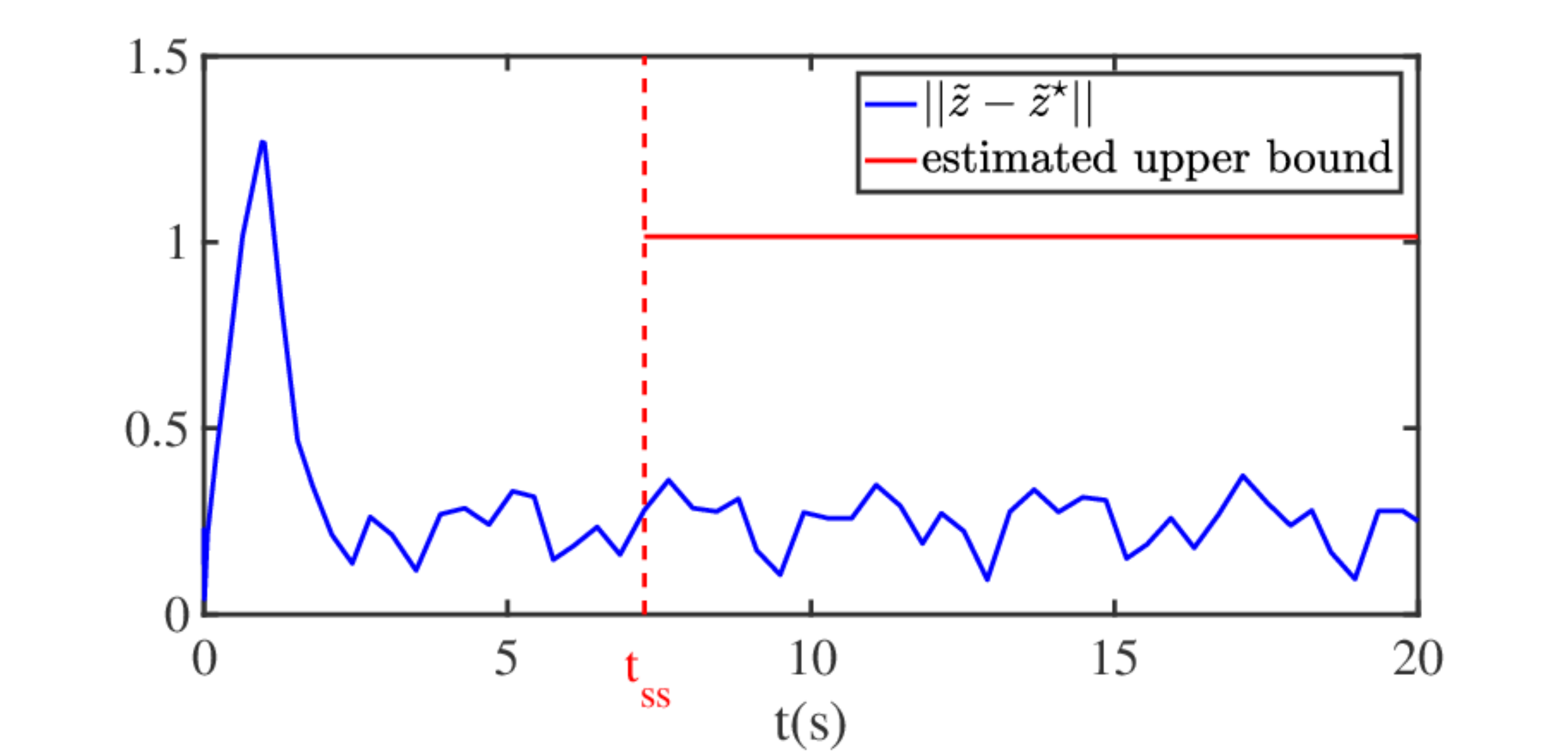}
	\caption{The long-term steady-state error is upper curbed by a bound (the red line) proportional to instantaneous optimum rate in Theorem \ref{theo:upbound}.}
		\label{fig:errorbound}
\end{figure}
In this section, we illustrate the applications of the derived bounds by considering the so-called Automatic Generation Control (AGC) of power systems \cite{Kundur,simpson2015secondary}. There has been a lot of interest recently in exploring the optimization perspective on the frequency control, see e.g., \cite{LowAGC, NaliPD} for review. As a proof of concept, here we look at a simplified AGC model designed to restore the frequency of the system. 



The simplified rescaled AGC model can be represented in PD form, with the Lagrangian
\begin{equation}  \label{eq:Lagrangian}
\ell = \frac{1}{2} \omega^T D \omega + \omega^T\left( B^{1/2}E^T p_\mathcal{E} - k u_{\mathrm{agc}}\right).
\end{equation}
We define $z = (\omega, p_\mathcal{E}, u_{\mathrm{agc}})$ wherein $\omega$ be the primal and $(p_\mathcal{E}, u_{\mathrm{agc}})$ be the dual variables. The variable vector $z$ represents the dimensionless physical states: frequency, line powers, and AGC efforts. Matrix $E$ describes the network topology, $B$ the rescaled line susceptances, $D$ the rescaled damping ratios, and $k$ the secondary control gains. The PD form $\dot{z} = \left(-\partial_{\omega}{\ell},\partial_{p_\mathcal{E}}{\ell},  \partial_{u_\mathrm{agc}}{\ell}\right)$ associated with the Lagrangian function \eqref{eq:Lagrangian} recovers frequency dynamics and the simplified AGC \cite{Kundur, SimpsonQuadratic}.

For the perturbed system, we consider an extension of the model with additional turbine delays modeled as $\dot{\hat{u}}_{\mathrm{agc}} = \frac{1}{T}(u_{\mathrm{agc}} - \hat{u}_{\mathrm{agc}})$. Such delays play an important role in frequency transients and should be considered in any practical studies. In this formulation, the frequency dynamics responds to the $\hat{u}_{\mathrm{agc}}$, rather than the AGC effort $u_{\mathrm{agc}}$. The resulting frequency equations becomes $\dot{\omega} = -\partial_{{\omega}}{\ell} + k (\hat{u}_\mathrm{agc} - u_\mathrm{agc})$.

We implement AGC on a SMIB (single-machine infinity-bus system) which consists of a generator and a purely inductive line. Further, we exert on the machine a sinusoidal exogenous torque which represents persistent load changes. Figure \ref{fig:errorbound} compares the actual response to the estimated bound.

\section{Concluding remarks}
In this work, we establish the strict contraction of the PD algorithm, applicable to a class of optimization problems that arise in network flows and distributed optimization. Strict contraction allows us to characterize the performance and robustness of the PD dynamics with respect to common deviations from the ``true'' PD dynamics. In particular, we consider the case of imperfect observers and also derive recursive bounds for hierarchical multi-scale optimization systems. 

While in this paper we restricted ourselves to systems with equality constraints, future work will include more general extensions to inequality constraints. Furthermore, saddle-node dynamics also appear naturally in Brayton-Moser potentials~\cite{ortega2003power, cavanagh2018transient}, which suggests an additional path for future research. Application-wise, we plan to explore the effect of inductive line delays~\cite{vorobev2017high, cavanagh2018transient}, on the performance of secondary controls~\cite{dorfler2016breaking} in microgrids. 

In recent years, the multiple-time-scale optimization perspective has also taken on increased importance in natural systems,  particularly in the context of systems biology~ \cite{kirschner2005plausibility,kashtan2005spontaneous, del2008modular}. In the brain, multiple interactions between different functional levels occur on a broad range of time-scales, involving weak or strong feedback interactions between genes, transcription factors, synapse formation, and global long-range connectivity~\cite{changeux2017climbing}. In such systems, ``general Darwinism”~\cite{changeux1973theory, edelman1987neural}  can play the role of an optimization criterion at every level~\cite{changeux2017climbing}, constrained by factors such as energy availability. Applying the tools developed in this paper to system modeling in such contexts is an additional subject for further research.


\bibliographystyle{IEEEtran}        
\bibliography{main}
\end{document}